\documentclass[a4paper,10pt]{amsart}

\usepackage{lmodern}

\usepackage[utf8]{inputenc}
\usepackage[english]{babel}

\usepackage{amssymb,amsfonts,amsthm,amsmath,latexsym}
\usepackage{enumerate, caption, array, esint}
\usepackage{hyperref, graphicx, color}
\usepackage{pgfplots}

\usepackage[margin=1.1in]{geometry}

\theoremstyle{plain}
\newtheorem{theorem}{Theorem}[section]
\newtheorem{corollary}[theorem]{Corollary}

\newtheorem{lemma}[theorem]{Lemma}
\newtheorem{proposition}[theorem]{Proposition}

\renewcommand{\mod}[1]{\ ({\rm mod\ }#1)}

\theoremstyle{definition}

\theoremstyle{remark}
\newtheorem*{remark}{Remark}

\newcommand\numberthis{\stepcounter{equation}\tag{\theequation}}

\newcommand{\ssum}[1]{\sum_{\substack{#1}}}
\newcommand{\e}{{\rm e}}
\newcommand{\dd}{{\rm d}}
\renewcommand{\epsilon}{\varepsilon}

\newcommand{\1}{{\mathbf 1}}
\newcommand{\R}{{\mathbb R}}
\newcommand{\N}{{\mathbb N}}

\newcommand{\C}{{\mathbb C}}

\newcommand{\cB}{{\mathcal B}}
\newcommand{\cE}{{\mathcal E}}

\newcommand{\cH}{{\mathcal H}}

\newcommand{\cM}{{\mathcal M}}

\newcommand{\vphi}{{\varphi}}
\newcommand{\Lt}{{\tilde \Lambda_r}}

\DeclareMathOperator{\Kl}{Kl}

\numberwithin{equation}{section}

\title{Sign changes of Kloosterman sums and exceptional characters}
\date{\today}

\author{Sary Drappeau}
\address{Aix Marseille Université, CNRS, Centrale Marseille \\ I2M UMR 7373 \\ 13453 Marseille \\France}
\email{sary-aurelien.drappeau@univ-amu.fr}

\author{James Maynard}
\address{Mathematical Institute \\ Radcliffe Observatory Quarter \\ Woodstock Road \\ Oxford OX2 6GG \\ United Kingdom}
\email{james.alexander.maynard@gmail.com}


\begin{document}

\begin{abstract}
We prove that the existence of exceptional real zeroes of Dirichlet $L$-functions would lead to cancellations in the sum~$\sum_{p\leq x} \Kl(1, p)$ of Kloosterman sums over primes, and also to sign changes of~$\Kl(1, n)$, where~$n$ runs over integers with exactly two prime factors. Our arguments involve a variant of Bombieri's sieve, bounds for twisted sums of Kloosterman sums, and work of Fouvry and Michel on sums of~$\left| \Kl(1, n)\right|$.
\end{abstract}

\maketitle{}

\section{Introduction}

\subsubsection*{Kloosterman sums}

For~$n\in\N_{>0}$ and a residue class~$a\mod{n}$, define the normalized Kloosterman sum as
$$ \Kl(a,n) = \frac1{\sqrt{n}} \ssum{\nu \mod{n} \\ (\nu, n)=1} \e\Big(\frac{\nu + a\bar{\nu}}n\Big), $$
where we write~$\e(z) = \e^{2\pi i z}$ and~$\bar{\nu}\nu \equiv 1\mod{n}$. These sums have a long history~\cite{Poincare, Kloosterman}, at the intersection of algebraic geometry and automorphic forms. The Weil bound~\cite{Weil, Matomaki} yields~$\left|\Kl(a,n)\right| \leq 2^{\omega(n)}$ if $32\nmid n$ and $\left|\Kl(a,n)\right|\le 2^{\omega(n)+1/2}$ in general, where~$\omega(n)$ is the number of distinct prime factors of~$n$. In particular, for a prime $p$
$$ \left|\Kl(1, p)\right|\leq 2. $$

Let~$\theta_{a, p}\in[0, \pi]$ be such that~$\Kl(a, p) = 2\cos(\theta_{a, p})$. The ``vertical'' Sato-Tate law, due to Katz~\cite{Katz}, asserts that the numbers
$$ \{ \theta_{a, p} \mid 1\leq a < p\} $$
become equidistributed, as~$p\to\infty$, with respect to the Sato-Tate measure~$\frac2\pi \sin(\theta)^2 \dd\theta$. The ``horizontal'' Sato-Tate conjecture is the claim that the numbers
$$ \{ \theta_{1, p} \mid p\leq x\} $$
become equidistributed with respect to the same measure, as~$x \to \infty$. This would of course imply that
$$ \sum_{p\leq x} \Kl(1, p) = o(\pi(x)) \qquad (x\to\infty). $$
Unfortunately the horizontal Sato-Tato conjecture is still open, and very little is known about this sum. Fouvry and Michel~\cite{FM-Changement} have obtained significant partial progress on replacing primes by almost-primes: they show that 
\[ \sum_{\substack{n<x \\ p|n\Rightarrow p>x^{1/23.9}}}(\left|\Kl(1,n)\right| \pm \Kl(1,n)) \gg \frac{x}{\log{x}}.
\]
In particular, it follows that there are infinitely many sign changes in the set~$\{\Kl(1, n), \omega(n)\leq 23\}$. After further work by many authors~\cite{FM-Bombieri, JSF-1, JSF-2, Matomaki, Xi-1}, the best know current result is due to Xi~\cite{Xi-2} and shows that there are infinitely many sign changes in the set~$\{\Kl(1, n), \omega(n)\leq 7\}$. We refer to the recent preprint~\cite{Xi-3} for more references and related questions.

\subsubsection*{Landau-Siegel zeroes}

In this work we study the implications of the existence of Landau-Siegel zeroes on this question. For~$D\geq 3$ an integer and~$\chi\mod{D}$ a real primitive character, define the value
$$ \eta_\chi := L(1, \chi) \log D. $$
It would follow from GRH that~$\eta_\chi \gg (\log D)/\log\log D$ and so $\eta_\chi\rightarrow\infty$ as $D\rightarrow \infty$. Unconditionally, however, Siegel's non-effective lower-bound~\cite{Siegel} $\eta_\chi \gg_\epsilon D^{-\epsilon}$ remains unsurpassed. It is as yet not even known if~$\eta_\chi$ is bounded away from~$0$; this is equivalent to the non-existence of real zeroes of~$L(s, \chi)$ close to~$1$. On the other hand, if there were a sequence of characters $\chi_1\mod{D_1}, \chi_2\mod{D_2},\dots$ such that $\eta_{\chi_i}\rightarrow 0$, then many desirable consequences would follow: the existence of twin primes~\cite{HB}, equidistribution of primes in arithmetic progressions to large moduli~\cite{FI-AP}, primes in very short intervals~\cite{FI-intervals}, or prime values of discriminants of elliptic curves~\cite{FI-discr-1,FI-discr}, for example.

We are interested in obtaining cancellations in the sum
$$ \ssum{p\le x}\Kl(1,p) $$
which go beyond the bound $2\pi(x)$ implied by the Weil bound should such exceptional characters exist.
\begin{theorem}\label{thm:MainTheorem}
Let $\epsilon>0$. Then there are constants $A,B>0$, depending only on $\epsilon$, such that for $D\ge 3$, $x\ge D^A$ and any primitive real character $\chi\mod{D}$, we have
\[
\Big| \sum_{p<x}\Kl(1,p) \Big|\le \pi(x)\Bigl(\epsilon+BL(1,\chi)\log{x}\Bigr).
\]
\end{theorem}
This statement is unconditional, but is only non-trivial if the value~$\eta_\chi$ is suitably small. We note that if there is a sequence of characters with $\eta_{\chi_i}\rightarrow 0$, then Theorem \ref{thm:MainTheorem} shows that for a suitable sequence of values of $x$,
\[
\sum_{p<x}\Kl(1,p)=o(\pi(x)),
\]
as predicted by the horizontal Sato-Tate conjecture. 

Unfortunately we do not know unconditionally the expected lower bound
\begin{equation}\label{eq:LowerBound}
\sum_{p<x}\left|\Kl(1,p)\right|\gg \pi(x).
\end{equation}
In particular, even if there was a sequence of characters with $\eta_{\chi_i}\rightarrow 0$, we would not be able to conclude from Theorem \ref{thm:MainTheorem} that there are even infinitely many sign changes in the sequence $\Kl(1,p)$. If instead of considering primes we consider products of exactly two primes, then the equivalent lower bound to \eqref{eq:LowerBound} is known thanks to work of Fouvry-Kowalski-Michel \cite{FKM}. For technical reasons, when working with products of two primes we consider the variant
$$ S(x) = \ssum{p, q\text{ prime}} \phi\Big(\frac{pq}x\Big) \log(pq) (\log p)(\log q) \Kl(1, pq) $$
where~$\phi:\R_+ \to \C$ is a smooth function compactly supported inside~$\R_+^*$. We note that the Weil bound implies unconditionally that $S(x)\ll x\log^2{x}$, whilst a variant of the horizontal Sato-Tate conjecture would suggest that we should have $S(x)=o_{\phi}(x\log^2{x})$. 
\begin{theorem}\label{thm:TwoPrimes}
Let~$\epsilon>0$. Then there exist~$A, B>0$, depending at most on~$\epsilon$, such that for any~$D\geq 3$, $x\ge D^A$ and any primitive real character~$\chi\mod{D}$, we have
\begin{equation}
|S(x)| \ll_\phi x(\log x)^2 \big\{\epsilon + B L(1,\chi)\log x \big\}.\label{eq:majo-S-prop}
\end{equation}
The implied constant depends only on the function $\phi$.
\end{theorem}
As with Theorem \ref{thm:MainTheorem}, this is unconditional but non-trivial only if a sequence of exceptional characters exist. Thus, in the presence of Siegel zeros, we are able to establish infinitely many sign changes of~$\Kl(1,pq)$. 
\begin{corollary}
For some absolute constants~$A, c>0$, if~$\eta_\chi \leq c$, then every interval~$[x, 2x] \subset [D^{A}, D^{100A}]$ contains two numbers~$(n_1, n_2)$ with~$\omega(n_1)=\omega(n_2)=2$, and $\Kl(1, n_1) \Kl(1, n_2) < 0$.
\end{corollary}
\begin{proof}
Choose~$\phi$ to be real-valued with~$\phi \geq \1_{[1, 2]}$, and consider the unsigned sum
$$ A(x) = \ssum{p, q \\ x < pq \leq 2x} (\log pq) (\log p)(\log q) \left|\Kl(1, pq)\right|. $$
We have the following lower bound, due to Fouvry-Kowalski-Michel~\cite[Proposition 5.1]{FKM}:
\begin{equation}\label{eq:mino-A}
A(x) \gg x(\log x)^2 \qquad (x\geq 2).
\end{equation}
Comparing~\eqref{eq:majo-S-prop} and~\eqref{eq:mino-A} yields the claimed statement.
\end{proof}

\subsection*{Notations}

We denote by~$P^-(n)$ (resp. $P^+(n)$) the smallest (resp. largest) prime factor of~$n$, with the conventions~$P^-(1)=\infty$ and~$P^+(1)=1$.

\section{Outline}
If there is a character $\chi\mod{D}$ such that $L(1,\chi)$ is very small, then `most' primes $p\in[D^A,D^{100A}]$ have $\chi(p)=-1=\mu(p)$, for some suitable constant $A$. By multiplicativity, this means that the (poorly understood) Moebius function $\mu$ can be well-approximated by the character $\chi$ (which is periodic $\mod{D}$, and so better understood) on integers $n<D^{100A}$ with no small prime factors. In particular, for $n\le D^{100A}$
\begin{equation}\label{eq:approx-chi-mu}
\Lambda(n)=(\mu\ast\log)(n)\approx (\chi\ast\log)(n),
\end{equation}
and so we can approximate $\Lambda$ by the convolution of two simpler sequences. By applying the hyperbola method, this would allow us to estimate a sum $\sum_{n<x}\Lambda(n)a_n$ provided we could suitably estimate
\begin{equation}
\sum_{d\le Y_1}\chi(d)\sum_{\substack{n\le x\\ d|n}}a_n\log\frac{n}{d},\quad\text{and}\quad \sum_{d\le Y_2}\log{d}\sum_{\substack{n<x\\ d|n}}a_n\chi\Bigl(\frac{n}{d}\Bigr)
\label{eq:HyperbolaSums}
\end{equation}
for some choice of $Y_1,Y_2$ with $Y_1Y_2=x= D^{100A}$. Often one can suitably estimate such sums for $Y_1=Y_2=x^{1/2-\epsilon}$, which just falls short of this requirement. Much of the work on the distribution of primes under the assumption of a Siegel-Landau zero followed this strategy, and the key technical challenge is then to obtain a suitable estimate for one of the sums in \eqref{eq:HyperbolaSums} with $Y_1$ or $Y_2$ slightly beyond $x^{1/2}$. 

In our situation, $a_n=\Kl(1,n)$, and estimates for the two sums in \eqref{eq:HyperbolaSums} with $Y_1=Y_2=x^{1/2-\epsilon}$ are obtained in essentially the same way as Fouvry and Michel~\cite{FM-Changement}. Unfortunately we do not know how to extend this work beyond $x^{1/2}$, and so this strategy fails. However, in the convolution identity $\Lambda(n)=\sum_{d|n}\mu(d)\log(n/d)$ it is only terms with $d\in[x^{1/2-\epsilon},x^{1/2+\epsilon}]$ which we are unable to handle. We note the alternative identity
\[
\Lambda(n)=\sum_{d|n}\mu(d)\Bigl(\log\frac{\sqrt{n}}{d}\Bigr).
\]
The presence of the term $\log(\sqrt{n}/d)$ means we expect that terms with $d\approx \sqrt{n}$ to contribute less, and so we might hope that these central values would be negligible. This is a variant of the idea that Bombieri introduced in his asymptotic sieve \cite{Bombieri}, where terms in $\Lambda_2(n)=\sum_{d|n}\mu(d)\log^2(n/d)$ from $d<n^{1-\epsilon}$ could be handled by assumptions on equidistribution of congruence sums, and terms with $d\in[n^{1-\epsilon},n]$ could be bounded by virtue of the fact that $\log^2(n/d)$ was small in this range. 

Unfortunately, as in Bombieri's work, this strategy fails if we wish to count primes. To maintain the feature that the support is essentially restricted to numbers with no small prime factors one multiplies by a short sieve weight, which loses a factor $\epsilon^2$ from the two variables $d$ and $n/d$. This precisely cancels out the gains of a factor $\epsilon^2$ coming from the range of $d$ and from the size of $\log(\sqrt{n}/d)$. Whilst this issue might appear to be a technicality, at least in Bombieri's work this is an expression of the fundamental parity problem of sieve methods. If instead we counted with a weight involving a higher power of $\log(\sqrt{n}/d)$ (thereby counting products of a bounded number of primes), then this strategy can succeed.

In our case, we are interested in $a_n=\Kl(1,n)$. Although in general we expect the Weil bound $\left|\Kl(1,n)\right|\le 2^{\omega(n)}$ to be essentially sharp, for \textit{most} integers $n$ we expect $\left|\Kl(1,n)\right|$ is actually much smaller than this. Indeed, the horizontal Sato-Tate conjecture would predict that for any fixed~$a$, the average size of $\left|\Kl(a,p)\right|$ is $\frac{2}{\pi}\int_0^\pi 2|\cos(t)| \sin^2(t)dt=8/3\pi<1$. By multiplicativity, we might then expect $\left|\Kl(1,n)\right|\approx (8/3\pi)^{\omega(n)}$ on average over $n$. Fouvry and Michel \cite{FM, FMerratum} combined an argument of Hooley~\cite{Hooley} based on the identity $\Kl(1,ab)=\Kl(\overline{a^2},b)\Kl(\overline{b^2},a)$ (for coprime $a,b$), with the \textit{vertical} Sato-Tate law, to show unconditionally that on average the factor $2^{\omega(n)}$ can be indeed improved to $(8/3\pi)^{\omega(n)}$. Since $8/3\pi<1$, numbers with a larger number of prime factors contribute less to the problematic sums, and so there is less of a loss from being restricted to a short sieve weight. This ultimately allows us to win an additional factor of $(\epsilon^{1-8/3\pi})^2$ for these sums involving middle sized $d$, which is enough to conclude that such terms make a negligible contribution, and so we are able to bound $\sum_{p<x}\Kl(1,p)$.

\section{Preparatory Lemmas}

\subsection{Level of distribution for twisted Kloosterman sums}\label{sec:LevelOfDistribution}
Here we make precise the claim that the sums in \eqref{eq:HyperbolaSums} can be estimated with $Y_1=Y_2=x^{1/2-\epsilon}$ by a variation of the work of Fouvry-Michel \cite{FM-Changement}.
\begin{proposition}\label{prop:distrib-12}
Let~$\epsilon>0$ be fixed. There exists~$\eta = \eta(\epsilon)>0$, such that for any real~$x\geq 2$, positive integer~$D\leq x^\eta$ and all characters~$\chi\mod{D}$, we have
\begin{equation}
\ssum{q\leq x^{1/2 - \epsilon}} \Big| \ssum{n \equiv 0\mod{q}} \phi\Big(\frac nx\Big) \chi(n) \Kl(1,n)\Big| \ll_{\epsilon, \phi} x^{1-\eta}.\label{eq:distrib-12}
\end{equation}
\end{proposition}

\begin{remark}
Note that the case~$D=1$ of the previous statement is a weaker form of Proposition 2.1 of~\cite{FM-Changement}.
\end{remark}

\begin{proof}
We may plainly assume that the sums are restricted to~$(q, D)=(n, D)=1$. The bound we claim is a variant of Proposition 2.1 of~\cite{FM-Changement}, which is based on the Kuznetsov formula~\cite{Kuz, DI}, the Weil bound for Kloosterman sums, and a uniform bound~$\theta\leq 1/4-\epsilon$ towards Ramanujan-Petersson, which was first due to Luo-Rudnick-Sarnak~\cite{LRS}.

The difference in our case is the presence of the character. Recently, Blomer and Mili\'{c}evi\'{c}~\cite{BM} have succeeded in analysing such sums in the context of modular forms with non-trivial nebentypus; another argument was used in~\cite{D}, which is simpler for our purpose here. We will rely on work of Topacogullari~\cite{Topacogullari} to estimate the spectral sums.

In our case, we will use the notations and normalization described in section 4.1.2 of~\cite{D}. Our aim is to apply the Kuznetsov formula~\cite[Lemma~4.5]{D} for the group~$\Gamma_0(qD)$, nebentypus~$\bar{\chi}\mod{D}$, with cusps~$\infty$ and~$1/q$, and parameters~$m\gets D$, $n\gets 1$. For each~$q$ in the left-hand side of~\eqref{eq:distrib-12}, Lemma~4.3 of~\cite{D}, with the choice of scaling matrices (depending only on~$q$ and $D$) given there, yields
$$ \chi(n) \Kl(1, n) = \chi(q) \e(-\bar{q}/D) S_{\infty, 1/q}(D, 1 ; n\sqrt{D}). $$
Let~$\kappa\in\{0, 1\}$ be such that~$\chi(-1) = (-1)^\kappa$. The Kuznetsov formula with test function~$\psi(t) = \phi(4\pi/(tx)) \sqrt{4\pi/(tx)}$ yields
$$ \overline{\chi(q)}\e(\bar{q}/D) \ssum{n \equiv 0\mod{q}} \phi\Big(\frac nx\Big) \chi(n) \Kl(1,n) = \sqrt{x}\big\{ \cH + \cE + \cM\big\}, $$
where
$$ \cH = \ssum{k>\kappa \\ k\equiv \kappa \mod{2}} \sum_{f\in \cB_k(qD, \bar{\chi})} {\dot\psi}(k) \Gamma(k) \overline{\rho_{f,\infty}(D)} \rho_{f, 1/q}(1), $$
$$ \cE = \sum_{{\mathfrak c}\text{ sing.}} \frac1{4\pi} \int_{-\infty}^{\infty} \frac{{\tilde \psi}(t)}{\cosh(\pi t)}\overline{\rho_{{\mathfrak c},\infty}(D, t)} \rho_{{\mathfrak c}, 1/q}(1, t) \dd t, $$
$$ \cM = \ssum{f\in \cB(qD, \bar{\chi})} \frac{{\tilde\psi}(t_f)}{\cosh(\pi t_f)} \overline{\rho_{f,\infty}(D)} \rho_{f, 1/q}(1). $$
Split~$\cM = \cM_1 + \cM_2$ where~$\cM_1$ is the contribution of those~$f$ with~$t_f\in\R$, and~$\cM_2$ is the remainder contribution, which consists of~$f$ with~$t_f\in[-i/4, i/4]$. Consider first~$\cM_2$. Using the bound $|{\tilde \psi}(t_f)| \ll x^{2|t_f|}$, we obtain by the Cauchy-Schwarz inequality
$$ \cM_2 \ll \Big(\sum_{f\in \cB(qD, \bar{\chi})} x^{2|t_f|} |\rho_{f,\infty}(D)|^2 \Big)^{1/2} \Big(\sum_{f\in \cB(qD, \bar{\chi})} x^{2|t_f|} |\rho_{f, 1/q}(1)|^2\Big)^{1/2}. $$
We have~$q\sqrt{D} \leq x^{1/2-\epsilon + \eta} \leq x^{1/2}$, so that Lemma~2.9 of~\cite{Topacogullari} may be applied to both sums, which yields, as~$x\to \infty$,
\begin{equation}
\cM_2 \ll x^{o(1)} \Big(\frac{\sqrt{xD}}q\Big)^{4\theta} \Big\{1 + \frac{\sqrt{D}}q\Big\} \ll x^{2\eta} (x/q^2)^{2\theta}.\label{eq:kuz-bound-exc}
\end{equation}
The contribution of~$\cE$, $\cH$ and~$\cM_1$ is handled by similar standard arguments~\cite[page~267]{DI}, using instead Lemma 2.8 of~\cite{Topacogullari}. We obtain
\begin{equation}
\cH + \cE + \cM_1 \ll x^{o(1)} D^{1/2} \ll x^\eta \label{eq:kuz-bound-reg}
\end{equation}
as~$x\to \infty$.  Grouping our bounds~\eqref{eq:kuz-bound-reg} and~\eqref{eq:kuz-bound-exc}, we find
$$ \ssum{q\leq x^{1/2-\epsilon}} \Big| \ssum{n \equiv 0\mod{q}} \phi\Big(\frac nx\Big) \chi(n) \Kl(1,n)\Big| \ll_{\epsilon, \phi} x^{1 + \eta - \epsilon} + x^{1+2\eta-(1/2-2\theta)\epsilon}. $$
By work of Kim-Sarnak~\cite{Kim}, we are ensured that~$\theta \leq 7/64 <1/4$, and therefore our claimed bound follows if~$Q\leq x^{1/2-\epsilon}$ and~$0<\eta \leq 3\epsilon/16$.
\end{proof}

From Proposition~\ref{prop:distrib-12} and the convolution~$\mu^2(n) = \sum_{d^2|n} \mu(d)$, we deduce the analogous bound with~$n$ restricted to square-free numbers.
\begin{corollary}\label{cor:distrib-12-sqf}
In the setting and notations of Proposition~\ref{prop:distrib-12}, we have
\begin{equation}
\ssum{q\leq x^{1/2 - \epsilon}} \Big| \ssum{n \equiv 0\mod{q}} \mu^2(n) \phi\Big(\frac nx\Big) \chi(n) \Kl(1,n)\Big| \ll_{\epsilon, \phi} x^{1-\eta}\label{eq:distrib-12-sqf}
\end{equation}
for a possibly smaller value of~$\eta>0$, but depending on~$\epsilon$ at most.
\end{corollary}

\subsection{Sieve weights}
To control the terms around the central point $d\in[x^{1/2-\epsilon},x^{1/2+\epsilon}]$ we will introduce a short sieve weight to maintain the feature that our sums are essentially supported on integers free of small prime factors. This is also crucial for carrying out the approximation~\eqref{eq:approx-chi-mu} as effectively as possible. We recall the construction of the~$\beta$-sieve from~\cite{opera}, and a few of its relevant properties for our application. Let
$$ \theta(n) = \sum_{d\mid n} \xi_d$$
be an upper-bound~$\beta$-sieve of level~$y$ and dimension~$\kappa\geq 0$, for the primes less than~$z$. The parameters~$y$, $z$ and~$\kappa$ will be chosen later to be small powers of $x$ with the condition that~$z^2 \leq y$. (We will ultimately choose $z=x^{\epsilon/(\log{\epsilon})^2}$ and $y=x^{\epsilon}$.) In particular, we have~$\xi_1=1$ and $|\xi_d|\leq 1$; $(\xi_d)$ is supported on integers up to~$y$ free of prime factor~$>z$.
Moreover, for any multiplicative function~$f$ with~$f(p)\in[0, \kappa]$, we have
\begin{equation}
\sum_{d} \frac{\xi_d f(d)}{d} \ll \prod_{\substack{p\leq z}}\Big(1+\frac{f(p)}{p}\Big)^{-1}.\label{eq:fundlem}
\end{equation}
For some technical simplifications, we will work with the smoothed version
$$ \theta'(n) := \sum_{d|n} \xi_d \log\Big(\frac nd\Big) = \sum_{d|n} \theta\Big(\frac nd\Big) \Lambda(d). $$
Note that~$\theta'(n) = \log n$ if~$P^-(n)>z$.

The key property we use of $\theta'$ is the following estimate on averages of multiplicative functions weighted by the sieve weight.
\begin{lemma}
Let the multiplicative function~$f$ be supported on squarefree numbers, with~$0\leq f(p) \leq \kappa$. Then
\begin{equation}
\ssum{n\leq x} \theta'(n) f(n) \ll x \prod_{z<p\leq x}\Big(1 + \frac{f(p)}p\Big).\label{eq:bound-theta-div}
\end{equation}
\end{lemma}
\begin{proof}
Denote~$S$ the sum on the left-hand side of~\eqref{eq:bound-theta-div}. By definition of~$\theta'$, we have
$$ S = \ssum{n\leq x} \sum_{p|n} f(n) \theta\Big(\frac{n}p\Big) \log p \leq \ssum{n\leq x} f(n) \theta(n) \sum_{p\leq x/n} f(p) \log p $$
by positivity of the summand. Therefore,
$$ S \ll_\kappa x\ssum{n\leq x} \frac{\theta(n) f(n)}n \leq x \ssum{P^+(n)\leq x} \frac{\theta(n) f(n)}n . $$
Opening the convolution in~$\theta(n)$, we arrive at
\begin{equation}\label{eq:computation-theta-euler}
\begin{aligned}
S {}& \ll_\kappa x \sum_{P^+(d)\leq x} \frac{\xi_d f(d)}d \ssum{P^+(n)\leq x \\ (n, d)=1} \frac{f(n)}{n} \\
{}& = x \prod_{p\leq x}\Big(1 + \frac{f(p)}p\Big) \sum_{P^+(d)\leq x} \frac{\xi_d f(d)}{d}\prod_{p|d}\Big(1 + \frac{f(p)}p\Big)^{-1}.
\end{aligned}
\end{equation}
By~\eqref{eq:fundlem}, the sum over~$d$ above is~$\ll_\kappa \prod_{p\leq z}(1 + \frac{f(p)}p)^{-1}$, and our claimed bound follows.
\end{proof}

\section{Argument for the upper bound}

In what follows, we let~$\epsilon>0$ be fixed, and~$\eta>0$ be a small parameter, to be chosen in terms of~$\epsilon$. We assume that~$D\leq x^\eta$.

\subsection{Initial setup}
For a positive integer $r$, define the arithmetic functions~$\Lt(n)$ by
$$ \Lt(n) = \sum_{d|n} \mu(d) \Big(\log\frac{\sqrt{n}}d\Big)^r. $$
We note that $\tilde{\Lambda}_1(n)=\Lambda(n)$ and $\tilde{\Lambda}_2(n)=\Lambda_2(n)-\Lambda(n)\log{n}$. On squarefree $n$, $\tilde{\Lambda}_1$ is supported on primes, and $\tilde{\Lambda}_2$ is supported on products of exactly two primes.

By partial summation, to show for all $\epsilon>0$ that
\[
\Bigl|\sum_{p\le x}\Kl(1,p)\Bigr|\le \epsilon\pi(x) +O_\epsilon(L(1,\chi)\log{x}),
\]
it suffices to show for any smooth compactly supported $\phi$ and any $\epsilon>0$ that
\[
\Bigl|\sum_{n}\mu^2(n)\phi\Bigl(\frac{n}{x}\Bigr)(\log{n})\tilde{\Lambda}(n)\Kl(1,n)\Bigr|\ll_\phi \epsilon x\log{x}+O_{\epsilon}(x L(1,\chi)\log^2{x}).
\]
Since $\mu^2(n)\tilde{\Lambda}(n)$ is only supported on primes, we have
\[
\sum_{n}\mu^2(n)\phi\Bigl(\frac{n}{x}\Bigr)(\log{n})\tilde{\Lambda}(n)\Kl(1,n)=\sum_{(n,D)=1}\mu^2(n)\phi\Bigl(\frac{n}{x}\Bigr)\theta'(n)\tilde{\Lambda}(n)\Kl(1,n)+O_\epsilon(x^\epsilon).
\]
For products of two primes, we note that for all~$x\geq 2$,
\begin{equation}\label{eq:Lt-pq}
\ssum{n\leq x \\ n\neq pq \text{ for all }p\neq q} \tilde{\Lambda}_2(n) \ll x,
\end{equation}
with the main contribution arising from pairs~$n=pq^2$, while
\begin{equation}\label{eq:Lt-D}
\ssum{n\leq x \\ (n, D)>1} \tilde{\Lambda}_2(n) \ll x\log\log x
\end{equation}
for~$2\leq D \leq x$. Thus
\begin{align*}
 \sum_{p, q} \phi\Big(\frac{pq}x\Big)& (\log pq) (\log p)(\log q) \Kl(1, pq) \\
={}& \sum_{p, q} \phi\Big(\frac{pq}x\Big) \theta'(pq) (\log p)(\log q) \Kl(1, pq) + O(x(\log x)(\log z)) \\
={}& \frac12 \sum_{(n, D)=1} \mu^2(n) \phi\Big(\frac nx\Big) \theta'(n) \tilde{\Lambda}_2(n) \Kl(1, n) + O(x(\log x)(\log z)). \numberthis\label{eq:Lt-sieved}
\end{align*}
We have used both~\eqref{eq:Lt-pq} and~\eqref{eq:Lt-D} in the last line.

Thus, we see that to establish Theorem \ref{thm:MainTheorem} and Theorem \ref{thm:TwoPrimes} it is sufficient to show for every $\epsilon>0$, every smooth compactly supported $\phi$ and each $r\in\{1,2\}$ that
\begin{equation}\label{eq:Target}
S(x)=S(x,r)=\sum_{(n,D)=1}\mu^2(n) \phi\Big(\frac nx\Big) \theta'(n) \Lt(n) \Kl(1, n) \ll_\phi \epsilon x\log^{r}{x}+O_\epsilon(xL(1,\chi)\log^{r+1}{x}).
\end{equation}
Let
$$ \nu(n) = (\mu \ast \mu \chi)(n) = \sum_{n = uv} \mu(u)\mu(v) \chi(v). $$
We write
\begin{equation}
\Lt(n) = \frac1{2^r} \sum_{n = abc} \nu(b) \chi(a) \Big(\log \frac{ab}{c}\Big)^r,\label{eq:Lt-convolution}
\end{equation}
so that for~$r\in\{1, 2\}$,
\begin{equation}\label{eq:Sx-abc}
S(x) = \frac1{r2^r} \ssum{(a, b, c) \in \N^3 \\ (abc, D)=1} \mu^2(abc) \phi\Big(\frac{abc}x\Big) \theta'(abc) \nu(b)\chi(a) \Big(\log\frac{ab}c\Big)^r \Kl(1, abc).
\end{equation}
We split the sum over~$(a, b, c)$ into $b>1$ or $b=1$ and $\min\{a,c\}\le x^{1/2}/y^3$, or $b=1$ and $\min\{a,c\}>x^{1/2}/y^3$. This gives
\begin{equation}\label{eq:decomp-S}
r2^r S(x) = S_L(x) + S_N(x) + S_C(x),
\end{equation}
with
\begin{equation}\label{eq:def-SL}
S_L(x) = \ssum{(a, b, c) \in \N^3 \\ b>1 \\ (abc, D)=1} \mu^2(abc) \phi\Big(\frac{abc}x\Big) \theta'(abc) \nu(b)\chi(a) \Big(\log\frac{ab}c\Big)^r  \Kl(1, abc),
\end{equation}
\begin{equation}\label{eq:def-SP}
S_N(x) = \ssum{(a, c) \in \N^2 \\ \min\{a, c\} \leq \sqrt{x}/y^3 \\ (ac, D)=1} \mu^2(ac) \phi\Big(\frac{ac}x\Big) \theta'(ac) \chi(a) \Big(\log\frac{a}c\Big)^r  \Kl(1, ac),
\end{equation}
\begin{equation}\label{eq:def-SM}
S_C(x) = \ssum{(a, c) \in \N^2 \\ \min\{a, c\} > \sqrt{x}/y^3 \\ (ac, D)=1} \mu^2(ac) \phi\Big(\frac{ac}x\Big) \theta'(ac) \chi(a) \Big(\log\frac{a}c\Big)^r \Kl(1, ac).
\end{equation}
We can control $S_L(x)$ by $L(1,\chi)$, since $\nu$ is supported on integers with all prime factors satisfying $\chi(p)=1$ and $\theta'$ is concentrated on integers free of small prime factors. Thus if $L(1,\chi)$ is small, we expect the support to be a lacunary sequence and so the sum $S_L(x)$ to be small. 

We can control $S_N(x)$ by the level of distribution estimates of Section \ref{sec:LevelOfDistribution}, since the sum over one of $a$ or $c$ is a long sum. 

Finally, $S_C$ is the contribution near the central point, and we will show this is small since $|\log(a/c)|\le 6\log{y} + O_\phi(1)$, which is small compared with $\log{x}$, $a$ is resticted to a short range on the logarithmic scale, and $\Kl(1,n)$ is typically small on numbers with many prime factors.

\subsection{Bounding the lacunary sum $S_L$}
We begin by bounding $S_L(x)$ in terms of $L(1,\chi)$.
\begin{lemma}\label{lem:Lacunary}
Let $D^8\le x^{\epsilon^2}\le z\le y^{1/(\log{\epsilon})^2}$. Then we have
\[
S_L(x)\ll x(\log{x})^{r}\Bigl( \epsilon^{-\kappa-2} L(1,\chi)\log{x}+\epsilon\Bigr).
\]
\end{lemma}
\begin{proof}
We note that $|\nu(n)|\ll (1\ast\chi)(n)$ and we recall that $\theta'(n)=(\Lambda\ast\theta)(n)$. Thus, bounding the summand of $S_L(x)$ and letting $d=ac$, and then expanding the definition of $\theta;$, we find
\begin{align*}
S_L(x)&\ll (\log{x})^r \sum_{\substack{db\ll x\\ b>1}}(1\ast\chi)(b)\theta'(bd)\tau(bd)^2\mu(bd)^2\\
&= (\log{x})^r \sum_{\substack{bef\ll x\\ b>1}}(1\ast\chi)(b)\theta(be)\Lambda(f)\tau(bef)^2\mu(bef)^2.
\end{align*}
When $f>z$ is prime we see $\theta(be)=\theta(bef)$. Thus the contribution from $f>z$ can be bounded by
\begin{equation}
(\log{x})^{r+1}\sum_{\substack{bg\ll x\\ b>1}}(1\ast\chi)(b)\theta(bg)\tau(bg)^4.\label{eq:a>z}
\end{equation}
The contribution from $f<z$ can be bounded by
\begin{equation}
(\log{x})^r \sum_{f\ll z}\Lambda(f)\sum_{\substack{be\ll x/f\\ b>1}}(1\ast\chi)(b)\theta(be)\tau(be)^2.\label{eq:a<z}
\end{equation}
These are precisely sums of the type considered in Section 24.7 of~\cite{opera}. Indeed, there it is established \cite[equation (24.58)]{opera} that if $D^8\le z\le y\le Y^{1/8}$ and $\rho$ is a completely multiplicative function with $0\le \rho(p)\le \kappa/6$ then
\[
\sum_{\substack{ab<Y\\ b>1}}(1\ast\chi)(b)\theta(ab)\rho(ab)\ll \frac{Y}{\log{Y}}\Bigl(L(1,\chi)\log{Y}+e^{-t}\Bigr)\Bigl(\frac{\log{Y}}{\log{z}}\Bigr)^{\kappa/2+1},
\]
where $t=\log{y}/\log{z}$. Applying this to the bounds \eqref{eq:a>z} and \eqref{eq:a<z} above, we find that provided $D^8\le z\le y\le x^{1/9}$ and $\kappa>6\times 2^4$, we have
\begin{align*}
(\log{x})^{r+1}\sum_{\substack{bg\ll x\\ b>1}}(1\ast\chi)(b)\theta(bg)\tau(bg)^4
&\ll x(\log{x})^r\Bigl(L(1,\chi)\log{x}+e^{-t}\Bigr)\Bigl(\frac{\log{x}}{\log{z}}\Bigr)^{\kappa/2+1},\\
(\log{x})^r\sum_{f\ll z}\Lambda(a)\sum_{\substack{be\ll x/f\\ b>1}}(1\ast\chi)(b)\theta(be)\tau(be)^2
&\ll (\log{x})^{r-1}\Bigl(L(1,\chi)\log{x}+e^{-t}\Bigr)\Bigl(\frac{\log{x}}{\log{z}}\Bigr)^{\kappa/2+1}\sum_{f\ll z}\frac{\Lambda(f)}{f}\\
&\ll x(\log{x})^r\Bigl(L(1,\chi)\log{x}+e^{-t}\Bigr)\Bigl(\frac{\log{x}}{\log{z}}\Bigr)^{\kappa/2}.
\end{align*}
Putting this together, we find
\[
S_L(x)\ll x(\log{x})^r\Bigl(L(1,\chi)\log{x}+e^{-t}\Bigr)\Bigl(\frac{\log{x}}{\log{z}}\Bigr)^{\kappa/2+1}.
\]
If $x^{\epsilon^2}\le z\le y^{1/(\log{\epsilon})^2}$, then this gives
\begin{equation}\label{eq:bound-l-final}
S_L(x)\ll x(\log{x})^r\Bigl( \epsilon^{-\kappa-2} L(1,\chi)\log{x}+\epsilon\Bigr),
\end{equation}
as required.
\end{proof}
\subsection{Bounding the non-central sum $S_N$}
\begin{lemma}\label{lem:NonCentral}
Let $y\ge x^\epsilon$. Then there is a constant $\eta>0$ depending only on $\epsilon$ such that
\[
S_N(x) \ll_\epsilon x^{1-\eta}.
\]
\end{lemma}
\begin{proof}
Denote by~$\chi_0(n)$ the principal character~$\mod{D}$. Opening the summation in~$\theta'(ac)$, we obtain
\begin{align*}
S_N(x) \leq {}& \sum_{d\leq y}\sum_{a<\sqrt{x}/y^3} \Big|\sum_{n\equiv 0\mod{[a, d]}} \mu^2(n) \phi\Big(\frac nx\Big)  \chi_0(n)  \Big(\log\Big(\frac{a^2}n\Big)\Big)^r \Big(\log\frac nd\Big) \Kl(1, n)\Big| \\ {}& \quad + \sum_{d\leq y} \sum_{c<\sqrt{x}/y^3} \Big|\sum_{n\equiv 0\mod{[c, d]}} \mu^2(n) \phi\Big(\frac nx\Big) \chi(n) \Big(\log\Big(\frac{c^2}n\Big)\Big)^r  \Big(\log\frac nd\Big) \Kl(1, n)\Big| \\
\ll {}& \sup_{\chi_1\mod{D}} \sup_{0\leq j \leq r+1} \sum_{q\leq \sqrt{x}/y} \tau(q)(\log x)^{r+1}\Big|\sum_{n\equiv 0\mod{q}} \mu^2(n) \phi\Big(\frac nx\Big) \chi_1(n) (\log(n/x))^j\Kl(1, n) \Big| \numberthis\label{eq:pcp-triangle}
\end{align*}
Recalling that we assume $y\ge x^\epsilon$ and applying Corollary~\ref{cor:distrib-12-sqf}, we obtain
\begin{equation}
S_N(x) \ll_\epsilon x^{1-\eta}\label{eq:bound-p-final}
\end{equation}
for some~$\eta$ depending at most on~$\epsilon$. 
\end{proof}

\subsection{Bounding the central point sum $S_C$}
\begin{lemma}\label{lem:Central}
Let $x^{\epsilon^2}\le z \le y^{1/(\log{\epsilon})^2}$ and $y\le x^{2\epsilon}$. Then we have
\[
S_C(x)\ll \epsilon^{10} x(\log{x})^r+\epsilon^{3/10}x(\log{x})^r\Bigl(\frac{\log{y}}{\log{z}}\Bigr)^2.
\]

\end{lemma}
\begin{proof}
Recall that the sum~$S_C(x)$ is defined in~\eqref{eq:def-SM}.  By construction, for all~$(a, c)$ in the summation range, we have~$|\log(a/c)| \ll \log y$.
Therefore, by the triangle inequality and the support condition on~$\phi$, our task is to bound
$$ (\log y)^r \ssum{x/y^3 \leq a \ll x y^3 \\ ac \asymp x} \mu^2(ac) \theta'(ac) \left|\Kl(1, ac)\right|. $$
Following an idea of Hooley~\cite{Hooley}, we let a parameter~$u\in[y, x^{1/4}]$ be given, and decompose accordingly
$$ a = a_1a_2, \qquad P^+(a_1)\leq u, \qquad P^-(a_2) > u, $$
and similarly let $c=c_1c_2$. Notice that
\[
\theta'(ac) = \theta(a_1c_1)\log(a_2c_2) + \theta'(a_1c_1)\ll \theta(a_1c_1)\log{x}+\theta'(a_1c_1)
\]
 since $P^-(a_2c_2)>u>z$. We write
\begin{equation}\label{eq:sympart-decomp}
\ssum{a_1a_2c_1c_2 \ll x \\ x^{1/2}/y^3<a_1a_2\ll x^{1/2}y^3 \\ P^+(a_1c_1)\leq u< P^-(a_2c_2)} \mu^2(a_1a_2c_1c_2) \theta'(a_1a_2c_1c_2)  \left|\Kl(1, a_1a_2c_1c_2)\right| = S_1 + S_2,
\end{equation}
where~$S_1$ is the contribution of~$a_1c_1\leq x^{1/10}$, and $S_2$ is the contribution of~$a_1c_1>x^{1/10}$.

Concerning~$S_1$, we have
$$ S_1 \ll \ssum{a_1c_1\leq x^{1/10} \\ P^+(a_1c_1)\leq u} \mu^2(a_1c_1) \{\theta(a_1c_1)\log x + \theta'(a_1c_1)\} T(a_1, c_1), $$
with
$$ T(a_1,c_1) = \ssum{a_2c_2 \ll x/a_1c_1 \\ P^-(a_2c_2)>u \\ x^{1/2}/y^3<a_1a_2\ll x^{1/2}y^3} \mu^2(a_2c_2) \left|\Kl(1, a_1a_2c_1c_2)\right|. $$
By a reasoning identical to page~277 of~\cite{FM}, for~$a_1c_1\leq x^{1/10}$, we have
\begin{align*}
T(a_1,c_1) {}& \leq \ssum{b\mod{a_1c_1} \\ (b, a_1c_1)=1} \left|\Kl(b^2, a_1c_1)\right| \ssum{a_2c_2 \ll x/a_1c_1 \\ x^{1/2}/a_1y^3<a_2\ll x^{1/2}y^3/a_1 \\ P^-(a_2c_2)>u \\ b a_2c_2 \equiv 1 \mod{a_1c_1}} \mu^2(a_2c_2) 2^{\omega(a_2c_2)} \\
{}& \ll \ssum{b\mod{a_1c_1} \\ (b, a_1c_1)=1} \left|\Kl(b^2, a_1c_1)\right| \frac{x}{a_1c_1\vphi(a_1c_1)} \frac{1}{\log x} \Big(\frac{\log x}{\log u}\Big)^2\ssum{x^{1/2}/a_1y^3<a_2\ll x^{1/2}y^3/a_1\\ P^-(a_2)>u}\frac{2^{\omega(a_2)}}{a_2} \\
& \ll  \frac{x}{a_1c_1\vphi(a_1c_1)} \frac{\log{y}}{\log^2 x} \Big(\frac{\log x}{\log u}\Big)^4 \ssum{b\mod{a_1c_1} \\ (b, a_1c_1)=1} \left|\Kl(b^2, a_1c_1)\right| \\
{}& = \frac{x \log y}{\log^2 x}\Big(\frac{\log x}{\log u}\Big)^4 \frac{2^{\omega(a_1c_1)}\kappa(a_1c_1)}{a_1c_1}
\end{align*}
for a multiplicative function~$\kappa$. Fouvry and Michel~\cite[equation~(3.4)]{FM} show that
$$ \kappa(p) = \frac{4}{3\pi} + O(p^{-1/4}), $$
which they deduce from Katz's Sato-Tate law for Kloosterman sums~\cite{Katz}. We insert this back into ~$S_1$, and we relax the summation conditions~$a_1c_1\leq x^{1/10}$. Letting
$$ f(n) := \mu^2(n) 4^{\omega(n)} \kappa(n), $$
we obtain
\begin{equation}\label{eq:sym-S1-f}
S_1 \ll \frac{x\log{y}}{\log^2 x}\Big(\frac{\log x}{\log u}\Big)^4 \sum_{P^+(n)\leq u} \frac{(\theta(n)\log x + \theta'(n)) f(n)}n.
\end{equation}
Note that
$$ f(p) = \frac{16}{3\pi} + O(p^{-1/4}) \qquad (p \leq u). $$
Taking Euler products, computations similar to~\eqref{eq:computation-theta-euler} yield
\begin{equation}
\sum_{P^+(n)\leq u} \frac{\theta(n) f(n)}{n} \ll \prod_{z<p\leq u} \Big(1 + \frac{f(p)}p\Big) \ll \Big(\frac{\log u}{\log z}\Big)^{16/(3\pi)}.\label{eq:bound-theta-f-harmon}
\end{equation}
On the other hand,
$$ \sum_{P^+(n)\leq u} \frac{\theta'(n)f(n)}n  = \ssum{q\leq u \\ q\text{ prime}} \frac{f(q)\log q}q \ssum{P^+(n)\leq u \\ (n, q)=1} \frac{\theta(n)f(n)}n \ll (\log u)\Big(\frac{\log u}{\log z}\Big)^{16/(3\pi)} $$
by dropping the condition~$(n,q)=1$ and using~\eqref{eq:bound-theta-f-harmon}. Inserting in~\eqref{eq:sym-S1-f}, we find
\begin{equation}\label{eq:S1Bound}
 S_1 \ll x \Bigl(\frac{\log{y}}{\log{x}}\Bigr)\Big(\frac{\log x}{\log u}\Big)^4 \Big(\frac{\log u}{\log z}\Big)^{16/(3\pi)}. 
\end{equation}
Consider now the contribution~$S_2$ to~\eqref{eq:sympart-decomp}. Using the Weil bound~$\left|\Kl(1, n)\right|\leq 2^{\omega(n)}$, we have
$$ S_2 \leq \sum_{n\ll x} \theta'(n) g_0(n), $$
where
$$ g_0(n) := \begin{cases} \mu^2(n) 4^{\omega(n)} &\text{ if }\prod_{p|n, p\leq u} p > x^{1/4}, \\ 0 & \text{ otherwise.} \end{cases} $$
By Rankin's inequality, we have~$g_0(n) \leq x^{-1/(4\log u)} g_1(n)$, where
$$ g_1(n) := \mu^2(n) 4^{\omega(n)} \prod_{p|n, p\leq u} p^{1/\log u}. $$
Note that the function~$g_1$ is multiplicative with
$$ g(p) = \begin{cases} 4 p^{1/\log u} & \text{ if } p \leq u \\ 4 & \text{ if } p>u. \end{cases} $$
In any case we have~$g(p)\leq 4e$. Assuming~$\kappa\geq 4e$, we obtain
\[
 S_2 \leq x^{-1/(4\log u)}\sum_{n\ll x} \theta'(n) g_1(n) \ll x \Big(\frac{\log x}{\log u}\Big)^4 \Big(\frac{\log u}{\log z}\Big)^{4e} \exp\Big\{-\frac{\log x}{4\log u}\Big\}. 
\]
We now choose $u=x^{1/(\log\epsilon)^2}$. Recalling that we have assumed $x^{\epsilon^2}\le z$, we see that this gives
\begin{equation}\label{eq:S2Bound}
S_2(x)\ll \epsilon^{10} x.
\end{equation}
Similarly, we substitute $u=x^{1/(\log\epsilon)^2}$ into \eqref{eq:S1Bound}. Using the assumption $z\le y\le x^{2\epsilon}$ and noting that $2-16/3\pi>3/10$, we find
\begin{align}
(\log{y})S_1
&\ll x\log{x}\Bigl(\log\frac{1}{\epsilon}\Bigr)^4\Bigl(\frac{\log{y}}{\log{x}}\Bigr)^{2-16/3\pi}\Bigl(\frac{\log{y}}{\log{z}}\Bigr)^{16/3\pi}\nonumber\\
&\ll \epsilon^{3/10} x\log{x}\Bigl(\frac{\log{y}}{\log{z}}\Bigr)^{2}.\label{eq:S1BoundFinal}
\end{align}
Recalling that $S_C(x)\ll (S_1+S_2)(\log{y})^r$, that $r\in\{1,2\}$, and inserting the bounds \eqref{eq:S1BoundFinal} and \eqref{eq:S2Bound}, we obtain
\[
S_C(x)\ll \epsilon^{10} x(\log{x})^r+\epsilon^{3/10}x(\log{x})^r\Bigl(\frac{\log{y}}{\log{z}}\Bigr)^2,
\]
as required.
\end{proof}

\subsection{Conclusion}
We choose
\[
y = x^\epsilon,\qquad z = x^{\epsilon/(\log\epsilon)^2}.
\]
With this choice, we see that $y$ and $z$ satisfy the assumptions of lemmas \ref{lem:Lacunary}, \ref{lem:NonCentral} and \ref{lem:Central}. Thus, provided $D^8\le x^{\epsilon^2}$ we may apply these Lemmas, giving
\begin{align*}
S(x)&\ll  S_L(x) + S_N(x) + S_C(x)\\
&\ll \epsilon^{-\kappa-2}x L(1,\chi)(\log{x})^{r+1}+\epsilon x(\log{x})^r+O_\epsilon(x^{1-\eta})+\epsilon^{3/10}x(\log{x})^r\Bigl(\log{\frac{1}{\epsilon}}\Bigr)^4
\end{align*}
for some quantity $\eta>0$ depending only on $\epsilon$. Reinterpreting~$\epsilon$, we obtain the claimed statement~\eqref{eq:Target}, which then gives Theorems \ref{thm:MainTheorem} and Theorem \ref{thm:TwoPrimes}.

\bibliography{kloosterman-siegel}

\providecommand{\bysame}{\leavevmode\hbox to3em{\hrulefill}\thinspace}
\providecommand{\MR}{\relax\ifhmode\unskip\space\fi MR }
\providecommand{\MRhref}[2]{%
  \href{http://www.ams.org/mathscinet-getitem?mr=#1}{#2}
}
\providecommand{\href}[2]{#2}
\begin{thebibliography}{FKM14}

\bibitem[BM15]{BM}
V.~Blomer and D.~Mili\'cevi\'c, \emph{Kloosterman sums in residue classes}, J.
  Eur. Math. Soc. (JEMS) \textbf{17} (2015), no.~1, 51--69.

\bibitem[Bom76]{Bombieri}
E.~Bombieri, \emph{On twin almost primes}, Acta Arith. \textbf{28} (1975/76),
  no.~2, 177--193.

\bibitem[DI83]{DI}
J.-M. Deshouillers and H.~Iwaniec, \emph{Kloosterman sums and {F}ourier
  coefficients of cusp forms}, Invent. Math. \textbf{70} (1982/83), no.~2,
  219--288.

\bibitem[Dra17]{D}
S.~Drappeau, \emph{Sums of {Kloosterman} sums in arithmetic progressions, and
  the error term in the dispersion method}, Proc. London Math. Soc.
  \textbf{114} (2017), no.~4, 684--732.

\bibitem[FKM14]{FKM}
\'E. Fouvry, E.~Kowalski, and P.~Michel, \emph{Algebraic trace functions over
  the primes}, Duke Math. J. \textbf{163} (2014), no.~9, 1683--1736.

\bibitem[FM03a]{FM-Bombieri}
\'E. Fouvry and P.~Michel, \emph{Crible asymptotique et sommes de
  {K}loosterman}, Proceedings of the {S}ession in {A}nalytic {N}umber {T}heory
  and {D}iophantine {E}quations, Bonner Math. Schriften, vol. 360, Univ. Bonn,
  Bonn, 2003, p.~27.

\bibitem[FM03b]{FM}
\bysame, \emph{Sommes de modules de sommes d'exponentielles}, Pacific J. Math.
  \textbf{209} (2003), no.~2, 261--288.

\bibitem[FM06]{FMerratum}
\bysame, \emph{Errata to the article: ``{S}ums of moduli of exponential sums''
  ({F}rench) [{P}acific {J}. {M}ath. {\bf 209} (2003), no. 2, 261--288;
  mr1978371]}, Pacific J. Math. \textbf{225} (2006), no.~1, 199--200.

\bibitem[FM07]{FM-Changement}
\bysame, \emph{Sur le changement de signe des sommes de {K}loosterman}, Ann. of
  Math. (2) \textbf{165} (2007), no.~3, 675--715.

\bibitem[FI03]{FI-AP}
J.~B. Friedlander and H.~Iwaniec, \emph{Exceptional characters and prime
  numbers in arithmetic progressions}, Int. Math. Res. Not. (2003), no.~37,
  2033--2050.

\bibitem[FI04]{FI-intervals}
\bysame, \emph{Exceptional characters and prime numbers in short intervals},
  Selecta Math. (N.S.) \textbf{10} (2004), no.~1, 61--69.

\bibitem[FI05]{FI-discr-1}
\bysame, \emph{The illusory sieve}, Int. J. Number Theory \textbf{1} (2005),
  no.~4, 459--494.

\bibitem[FI10]{opera}
\bysame, \emph{Opera de cribro}, American Mathematical Society Colloquium
  Publications, vol.~57, American Mathematical Society, Providence, RI, 2010.

\bibitem[FI13]{FI-discr}
\bysame, \emph{Exceptional discriminants are the sum of a square and a prime},
  Q. J. Math. \textbf{64} (2013), no.~4, 1099--1107.

\bibitem[HB83]{HB}
D.~R. Heath-Brown, \emph{Prime twins and {S}iegel zeros}, Proc. London Math.
  Soc. (3) \textbf{47} (1983), no.~2, 193--224.

\bibitem[Hoo64]{Hooley}
C.~Hooley, \emph{On the distribution of the roots of polynomial congruences},
  Mathematika \textbf{11} (1964), 39--49.

\bibitem[Kat88]{Katz}
N.~M. Katz, \emph{Gauss sums, {K}loosterman sums, and monodromy groups}, Annals
  of Mathematics Studies, vol. 116, Princeton University Press, Princeton, NJ,
  1988.

\bibitem[Kim03]{Kim}
H.~H. Kim, \emph{Functoriality for the exterior square of {${\rm GL}_4$} and
  the symmetric fourth of {${\rm GL}_2$}}, J. Amer. Math. Soc. \textbf{16}
  (2003), no.~1, 139--183, With appendix 1 by D. Ramakrishnan and appendix 2 by
  H. Kim and P. Sarnak.

\bibitem[Klo27]{Kloosterman}
H.~D. Kloosterman, \emph{On the representation of numbers in the form
  {$ax^2+by^2+cz^2+dt^2$}}, Acta Math. \textbf{49} (1927), no.~3-4, 407--464.

\bibitem[Kuz80]{Kuz}
N.~V. Kuznetsov, \emph{The {P}etersson conjecture for cusp forms of weight zero
  and the {L}innik conjecture. {S}ums of {K}loosterman sums}, Mat. Sb. (N.S.)
  \textbf{111(153)} (1980), no.~3, 334--383, 479.

\bibitem[LRS95]{LRS}
W.~Luo, Z.~Rudnick, and P.~Sarnak, \emph{On {S}elberg's eigenvalue conjecture},
  Geom. Funct. Anal. \textbf{5} (1995), no.~2, 387--401.

\bibitem[Mat11]{Matomaki}
K.~Matom\"aki, \emph{A note on signs of {K}loosterman sums}, Bull. Soc. Math.
  France \textbf{139} (2011), no.~3, 287--295.

\bibitem[Poi11]{Poincare}
H.~Poincar\'e, \emph{Fonctions modulaires et fonctions fuchsiennes}, Ann. Fac.
  Sci. Toulouse Sci. Math. Sci. Phys. (3) \textbf{3} (1911), 125--149.

\bibitem[Sie35]{Siegel}
C.~L. Siegel, \emph{{\"Uber die Classenzahl quadratischer Zahlk\"orper.}},
  {Acta Arith.} \textbf{1} (1935), 83--86.

\bibitem[SF07]{JSF-1}
J.~Sivak-Fischler, \emph{Crible \'etrange et sommes de {K}loosterman}, Acta
  Arith. \textbf{128} (2007), no.~1, 69--100.

\bibitem[SF09]{JSF-2}
\bysame, \emph{Crible asymptotique et sommes de {K}loosterman}, Bull. Soc.
  Math. France \textbf{137} (2009), no.~1, 1--62.

\bibitem[Top15]{Topacogullari}
B.~Topacogullari, \emph{On a certain additive divisor problem}, preprint
  (2015), arXiv, http://arxiv.org/abs/1512.05770v2.

\bibitem[Wei48]{Weil}
A.~Weil, \emph{On some exponential sums}, Proc. Nat. Acad. Sci. U. S. A.
  \textbf{34} (1948), 204--207.

\bibitem[Xi15a]{Xi-1}
P.~Xi, \emph{Sign changes of {K}loosterman sums with almost prime moduli},
  Monatsh. Math. \textbf{177} (2015), no.~1, 141--163.

\bibitem[Xi15b]{Xi-2}
\bysame, \emph{Sign changes of {K}loosterman sums with almost prime moduli.
  {II}}, Int. Math. Res. Notices (2015), to appear.

\bibitem[Xi18]{Xi-3}
\bysame, \emph{When {Kloosterman} sums meet {Hecke} eigenvalues}, preprint
  (2018), arXiv, http://arxiv.org/abs/1801.07658.v2.

\end{thebibliography}
\bibliographystyle{amsalpha2}

\end{document}